\documentclass{article}
\usepackage[english]{babel}
\usepackage[letterpaper,top=2cm,bottom=2cm,left=3cm,right=3cm,marginparwidth=1.75cm]{geometry}

\usepackage{amsmath, amsthm}
\usepackage{graphicx}
\usepackage[colorlinks=true, allcolors=blue]{hyperref}
\usepackage{amssymb,color}
\usepackage{caption}
\usepackage{subcaption}
\usepackage{morefloats}
\usepackage{tikz}
\usetikzlibrary{automata,positioning,arrows.meta}
\tikzset{
  box/.style={rectangle, draw, rounded corners, align=center, minimum width=3cm, minimum height=1cm},
  arrow/.style={-{Stealth}, thick},
}

\title{Existence and approximate controllability results for time-fractional stochastic Navier-Stokes equations}
\author{Renu Chaudhary, Simeon Reich and Juan J. Nieto}

\begin{document}
\maketitle

\numberwithin{equation}{section} 

\newtheorem{theorem}{Theorem}      
\newtheorem{lemma}{Lemma}          
\newtheorem{definition}{Definition}
\newtheorem{prop}{Proposition}     
\newtheorem{coro}{Corollary}       
\newtheorem{rmk}{Remark}           

\begin{center}
\author{Renu Chaudhary$^{a,b,*}$, Simeon Reich$^a$, Juan J. Nieto$^c$ \\
$^a$Department of Mathematics,
 The Technion — Israel Institute of Technology, 32000 Haifa, Israel\\
 $^b$ Faculty of Applied Natural Sciences and Humanities (FANG), Technical University of Applied Sciences Würzburg-Schweinfurt, 97421 Schweinfurt, Germany\\
$^c$CITMAga, Departamento de Estat\'{\i}stica, An\'alise Matem\'atica e Optimización
Universidade de Santiago de Compostela, 15782 Santiago de Compostela, Spain\\
$^*$Corresponding author email: renu.chaudhary@thws.de}
\end{center}

\begin{abstract}
This paper deals with time-fractional stochastic Navier-Stokes equations, which are characterized by the coexistence of stochastic noise and a fractional power of the Laplacian. We establish sufficient conditions for the existence and approximate controllability of a unique mild solution to time-fractional stochastic Navier-Stokes equations. Using a fixed point technique, we first demonstrate the existence and uniqueness of a mild solution to the equation under consideration. We then establish approximate controllability results by using the concepts of fractional calculus, semigroup theory, functional analysis and stochastic analysis.
\end{abstract}

\section{Introduction}
The origin of the theory of stochastic Navier-Stokes equations (SNSEs, for short) dates back to the work of Landau and Lifshitz \cite{landau} in their 1959 book entitled ``Fluid Mechanics''. Later on, after the pioneering work of Bensoussan and Temam \cite{Bensoussan} on the mathematical viewpoint of SNSEs, these equations have been widely studied and applied to characterize several phenomena in various domains, especially in the fields of fluid dynamics, gas dynamics, nonlinear acoustics, and traffic flow.  The stochasticity introduced in the well-known Navier-Stokes equation might help us comprehend physical processes and the mechanisms of fluid turbulence in a more sophisticated manner. The SNSEs provide a relatively simple method for investigating turbulence, distortion caused by laminar momentum transport, and the decay of dissipation layers formed as a result. In the last few years significant progress has been made in the mathematical and numerical analysis of SNSEs. In particular, Da Prato and Debussche \cite{Prato1} and Brze\'{z}niak \cite{Brzezniak1} studied two-dimensional SNSEs. Kim \cite{kim} established the existence of strong solutions to SNSEs in $R^3$. Recently, Hofmanova \cite{Hofmanova} has derived global existence and non-uniqueness results for three-dimensional SNSEs.

It is well known that fractional derivatives (rather than integer order derivatives) have proven to be highly useful in analyzing many real world problems because of their ability to represent long memory processes; see \cite{kilbas2006theory, podlubny1998fractional, Machado}.
Experiments show that fractional derivatives are among of the top tools for modeling anomalous diffusion processes and viscoelasticity; see, for instance, \cite{caputo, Kai, ding, singh}. Consequently, fractional order Navier-Stokes equations (FNSEs, for short) can be used to model more precisely anomalous diffusion in fractal media. As a result, numerous researchers have conducted in-depth studies regarding FNSEs. For example, Carvalho-Neto and Planas \cite{de} established the existence of mild solution to FNSEs in $R^n$; Momani and Odibat \cite{Momani} studied analytical solutions of FNSEs using different methods; Zhou and Peng \cite{zhou} established the existence and uniqueness of local and global mild solutions to FNSEs. Xu et al. \cite{xu} and Zou et al. \cite{zou1} have analyzed FNSEs driven by fractional Brownian motion. Recently, Han \cite{Han} studied $L_p$-solvability for stochastic time fractional Burgers’ equations driven by multiplicative space-time white noise. The author established conditions for the existence and uniqueness and Hölder regularity of the local and global solution.

Controllability plays an essential role in the design and analysis of control systems. Starting from the seminal work of Kalman \cite{Kalman} in 1963, the notion of controllability helps one to understand whether a particular dynamical control system can be controlled or not, and if it can, what the optimum control should be in order to obtain the intended outcome. Our understanding of the controllability of fractional control systems has greatly increased in recent years as more scientists have become interested in the topic of controllability and its applications in biological science, economics, aerospace engineering and electrical engineering; see, for instance, Barnett \cite{barnett1975introduction}, Curtain and Zwart \cite{curtain2012introduction}, Dauer and Mahmudov \cite{Dauer}, Karthikeyan et al. \cite{karthikeyan}, Sakthivel et al. \cite{sakthivel} and references therein. 
In spite of its being an important topic, only a few papers deal with the controllability of FNSEs. Xi et al. \cite{xi} established approximate controllability results for time-fractional Navier-Stokes equations involving time delay, and Liao et al. \cite{liao} have recently presented global controllability results for Navier-Stokes equations.  However, the existing literature does not address the question of existence and approximate controllability of FNSEs involving randomness and fractional powers of the Laplacian. Therefore it is quite natural and interesting to establish existence and controllability results for time-fractional SNSEs. 

This paper makes significant contributions to this unexplored domain by investigating the existence and approximate controllability of solutions to time-fractional SNSEs. We develop an innovative application of fixed-point theorems tailored to handle the intricacies of fractional stochastic systems, establishing the existence and uniqueness of mild solutions under conditions not previously addressed. By examining the approximate controllability within this complex setting, we provide new theoretical results that expand on existing controllability concepts, accounting for both fractional and random influences. Our results have practical implications for modeling and controlling fluid dynamics phenomena where anomalous diffusion and random effects are prevalent, such as pollutant dispersion in atmospheric sciences and intricately controlled flows in biotechnology.
We mainly focus on the following time-fractional SNSE. This SNSE is the simplest representative equation of the tri-interaction between wave steepening, small dissipation and random perturbations, which are represented as the nonlinearity, the fractional power of the Laplacian and the stochastic process, respectively, in a bounded domain $\mathcal{O}\subset R^d (1\le d\le 3)$ with a smooth boundary $\partial \mathcal{O}$:
\begin{eqnarray}\label{BE1}
\partial^{\eta}_{t} z(t,x) + \nu(-\Delta)^{\frac{\alpha}{2}} z(t,x)-(z(t,x)\cdot \nabla)z(t,x)-\nabla \rho(t,x)=Cv(t,x)+\hbar(t,z(t,x))\frac{dW(t)}{dt}, \,\,
\end{eqnarray}
where $(t,x)\in(0,T]\times \mathcal{O}$,
with the incompressibility condition
\begin{align}\label{BE11}
\nabla \cdot z(t,x)=0, \quad (t,x)\in(0,T]\times \mathcal{O},
\end{align}
the Dirichlet boundary condition
\begin{align}\label{BE12}
z(t,x)=0, \quad (t,x)\in(0, T]\times \partial \mathcal{O},
\end{align}
and the initial conditions
\begin{align}\label{BE13}
z(0,x)=z_0(x), \quad x\in \mathcal{O}.
\end{align}

Here $\eta\in(0,1)$, $z(t,x)$ denotes the velocity field at a point $x\in R^d$,  $\nu>0$ is the viscosity coefficient, $\rho$ denotes the associated pressure field, $C$ represents a linear operator, $v$ denotes the control function, $\hbar$ is a nonlinear function which represents the external force, $W(t)$ denotes a Wiener process, and the operator $(-\Delta)^{\frac{\alpha}{2}}$, $\alpha\in (1,2)$, denotes a fractional power of the Laplacian (see \cite{debbi2016,fraclaplace}).

The time-fractional stochastic Navier-Stokes equations formulated in this work serve as a robust framework for modeling fluid behavior in scenarios where classical assumptions of instantaneous response and deterministic evolution no longer hold. The fractional time derivative effectively captures memory-dependent dynamics, which are characteristic of viscoelastic or anomalously diffusing fluids, while the stochastic perturbations reflect uncertainties and fluctuations inherent in many natural and engineered systems. Such a model becomes particularly relevant in applications ranging from subsurface contaminant transport and atmospheric dispersion of pollutants to microfluidic flows in biomedical devices. In these settings, the combination of long-range temporal correlations and random influences plays a crucial role, and understanding the controllability of such systems can significantly inform the design of efficient monitoring, control, and optimization strategies.

\begin{figure}[h]
\centering
\begin{tikzpicture}[
  node distance=1.6cm,
  block/.style = {rectangle, draw, minimum width=2.2cm, minimum height=1cm, align=center, rounded corners},
  arrow/.style = {thick, -{Latex[length=2mm]}},
  font=\small
  ]

\node[block] (lemma15) at (-6,0) {Lemma 1-5};
\node[block] (lemma6) at (-1.5,0) {Lemma 6};
\node[block] (lemma7)  at (3.0,0)  {Lemma 7};
\node[block] (def12)     at (3,-2) {Defintion 1-2};

\node[block] (assu12)     at (-6,-2) {Assumption 1-2};
\node[block] (thm1) at (-1.5,-2) {Theorem 1};
\node[block] (assum3)  at (-6,-4)  {Assumption 3};

\node[block] (def3)  at (3,-4)  {Definition 3};

\node[block] (thm2) at (-1.5,-4)   {Theorem 2};

\draw[arrow] (lemma15) -- (lemma6);
\draw[arrow] (lemma6) -- (lemma7);

\draw[arrow] (assu12) -- (thm1);
\draw[arrow] (assum3) -- (thm2);
\draw[arrow] (lemma15) -- (thm1);
\draw[arrow] (def3) -- (thm2);
\draw[arrow] (lemma7) -- (thm1);
\draw[arrow] (def12) -- (thm1);
\draw[arrow] (assu12) -- (thm2);
\end{tikzpicture}
\caption{Logical structure of the paper: Lemmas, Definitions, Assumptions, and Theorems with explicit dependencies.}
\end{figure}
\section{Notations and preliminaries}
For $1 < p <\infty$, $L^p(\mathcal{O})$ denotes the Lebesgue space and $W^{k,p}(\mathcal{O})$ denotes the Sobolev spaces with $H^k(\mathcal{O}):=W^{k,2}(\mathcal{O})$.
For $p=2$, $L^2(\mathcal{O}) =: H$ denotes the Hilbert space with inner product $\langle \cdot, \cdot \rangle$.
 Let $C^\infty(\mathcal{O})$ denote the space of all infinitely differentiable functions and $C_0^\infty(\mathcal{O}):=\{z\in (C^\infty(\mathcal{O}))^d: \,\nabla \cdot z=0, \,\, z \,\,\text{has}\,\, \text{compact}\,\,\text{support}\,\, \text{in}\,\, \mathcal{O}\}$. Let $H_0^1(\mathcal{O})$ denote the closure of $C_0^\infty(\mathcal{O})$ in $(L^p(\mathcal{O}))^d$.

Let $(\Sigma,\mu,\{\mu_t\}_{t\geq0},\mathcal{P})$ be a filtered complete probability space with normal filtration $\{\mu_t\}_{t\geq0}$. The Wiener process $\{W(t):t\geq 0\}$ possesses a finite trace linear bounded covariance operator $Q\geq 0$ with $Tr(Q)=\sum_{m=1}^\infty \nu_m=\nu<\infty$ and $Qc_m=\nu_mc_m$, where $\{c_m,m\in \mathbb{N}\}$ denotes a complete orthonormal basis for $H$. If $\{\omega_m\}_{m\in \mathbb{N}}$ is a sequence
of one-dimensional Wiener processes, then
\begin{eqnarray*}
W(t)=\sum_{m=1}^\infty \sqrt{\nu_m}\,\omega_m(t)\,c_m,\quad t\geq 0.
\end{eqnarray*}
$L_0^2=L^2(Q^{\frac{1}{2}}(U),H)$ denotes the Hilbert space of Hilbert-Schmidt operators from $Q^{\frac{1}{2}}(U)$ to $H$ endowed with the norm
\begin{eqnarray*}
\|\psi\|_{L_0^2}=\bigg(\sum_{m=1}^\infty\|\psi\nu_m\|^2\bigg)^{\frac{1}{2}},\quad \psi\in L_0^2.
\end{eqnarray*}
Furthermore, let $L^p(\Sigma;H)$ be the Hilbert space of $H$-valued random variables with norm
\begin{eqnarray*}
\|z(\cdot)\|_{L^p(\Sigma;H)}=(E[\|z(\cdot)\|^p_H])^{1/p},
\end{eqnarray*}
 where $E$ represents the expectation with respect to the measure $\mathcal{P}$ defined by
 \begin{eqnarray*}
 E[\|z(\cdot)\|^p_H] := \int_{\Sigma}\|z(\omega)\|^p_{H}d\mathcal{P}(\omega)<\infty, \quad\omega\in \Sigma.
\end{eqnarray*}

Let $P$ be the Helmholtz projection operator defined on $(L^p(\mathcal{O}))^d$ with range $H_0^1(\mathcal{O})$. Let $A=-\nu P \Delta$ denote the Stokes operator in $H^{1}_{0}(\mathcal{O})$ with domain $D(A)=\{z\in H^{1}_{0}(\mathcal{O})\cap H^2(\mathcal{O}): \,\,z(t,x)=0 \,\, \forall \,\, (t,x)\in(0, T]\times \partial \mathcal{O}\}$.
Since $\mathcal{O}$ is bounded, the inverse $A^{-1}$ exists and is a compact operator on $H$. Furthermore, the fractional powers of A, that is, $A^{\frac{\alpha}{2}}=A_\alpha=\nu P (-\Delta)^{\frac{\alpha}{2}}$, can be defined by
\begin{align*}
A^{\frac{\alpha}{2}}v_m=e_m^{\frac{\alpha}{2}}v_m
\end{align*}
with domain
\begin{align*}
\mathcal{H}^\alpha=D(A^{\frac{\alpha}{2}})=\{z\in H:\|z\|^2_{\mathcal{H}^\alpha}=\sum_{m=1}^\infty e_m^{\frac{\alpha}{2}}u_m^2<\infty\},
\end{align*}
where $u_m=\langle z,v_m \rangle$. We set $\|z\|_{\mathcal{H}^\alpha} := \|A^{\frac{\alpha}{2}} z\|$ and consider the associated dual space ${\mathcal{H}^{-\alpha}}$ with the inverse operator $A^{-\frac{\alpha}{2}}$. The operator $-A_{\alpha}$ generates an analytic semigroup $S_\alpha(t)=e^{-tA_{\alpha}}$ of operators which are compact for $t>0$. The control function $v\in L^p_{\mu}([0,T], U)$, where $U$ is the separable Hilbert space. For simplicity of notation, the set of admissible controls is denoted by $U_{ad}=L^p_{\mu}([0,T], U)$. The mapping $C:U\rightarrow H$ is a bounded linear operator. We also consider the bilinear operator $G(z, w) := -P(z\cdot\nabla) w$ with $D(G) = H^1_0(\mathcal{O})$. In a slight abuse of notation, we write $G(z):= G(z, z)$.\\

Applying the Helmholtz-Hodge projection operator $P$ to the time-fractional SNSE (\ref{BE1})--(\ref{BE13}), we obtain the following abstract equation:
\begin{eqnarray}\label{BE2}
\begin{split}
 ^CD^{\eta} z(t)&=-A_{\alpha} z(t)+G(z(t))+Cv(t)+\hbar(t,z(t))\frac{dW(t)}{dt}, \quad t\in(0,T],\\
z(0)&=z_0,\\
\end{split}
\end{eqnarray}
where $^CD^{\eta}$ denotes the Caputo fractional derivative of order $\eta\in(0,1)$. Here $-A_{\alpha}$ is the infinitesimal generator of an analytic semigroup $\{S_{\alpha}(t):t\geq 0\}$. The initial value $z_0$ is an $H^\alpha$-valued $\mu_0$-measurable random variable which is independent of $W$. Instead of $Pv(t)$ and $P\hbar(t,z(t))$, we will use the notation $v(t)$ and $\hbar(t,z(t))$, respectively.


\begin{definition} \cite{podlubny1998fractional} The Caputo fractional derivative of $z\in C^m([0,T])$ of order $\eta$, $m-1<\eta\leq m$, is defined by
\begin{align*}
^C{D}^{\eta}_{0^+}z(t) := \frac{1}{\Gamma(m-\eta)}\int_0^t(t-r)^{m-\eta-1}\frac{d^m}{ds^m}z(r)dr,
\end{align*}
where $z^{(m-1)}$ is absolutely continuous in every compact interval $[0,T]$, $T>0$.
\end{definition}

\begin{lemma}\cite{zou}
For any $\alpha>0$, the operator $-A_\alpha$ generates an analytic semigroup $S_\alpha(t)=e^{-tA_\alpha}$, $t\geq 0$, on $L^p$.
Moreover, we have
\begin{align*}
\|A_\beta S_\alpha(t)\|_{\mathcal{L}(L^p)}\leq C_{\alpha,\beta}\,t^{-\frac{\beta}{\alpha}},\quad t>0,
\end{align*}
where $\beta\geq0$, the constant $C_{\alpha,\beta}>0$ depends on $\alpha$ and $\beta$, and $\mathcal{L}(L^p)$ denotes the Banach space of all bounded linear operators from $L^p$ into itself.
\end{lemma}

\begin{lemma}\cite{kruse} \label{Gundys}For any $p\geq2$, $0\leq s_1<s_2\leq T$, and predictable stochastic process $\chi:[0,T]\times \Sigma \rightarrow L_0^2$ such that
 \begin{align*}
E\bigg[(\int_0^T\|\chi(s)\|^2_{L_0^2}ds)^{\frac{p}{2}}\bigg]<\infty,
\end{align*}
we have
 \begin{align*}
E\bigg[\bigg\|\int_{s_1}^{s_2}\chi(s)dW(s)\bigg\|^{p}\bigg]\leq \kappa(p)E\bigg[\bigg(\int_{s_1}^{s_2}\|\chi(s)\|^2_{L_0^2}ds\bigg)^{\frac{p}{2}}\bigg].
\end{align*}
Here the constant
\begin{align*}
\kappa(p)=\bigg(\frac{p}{2}(p-1)\bigg)^{\frac{p}{2}}\bigg(\frac{p}{p-1}\bigg)^{p(\frac{p}{2}-1)}.
\end{align*}
\end{lemma}

\begin{definition}\cite{zou}
A $\mu_t$-adapted stochastic process $\{z(t)\}_{t\in[0,T]}$ is said to be a mild solution of $(\ref{BE2})$ if for each $v\in U_{ad}$, 
$\{z(t)\}_{t\in[0,T]}\in C([0,T],\mathcal{H}^\alpha)$ $\mathcal{P}$-a.s., and
\begin{align}\label{mild}
 z(t)=&M_{\eta}(t)z_0+\int_0^t(t-r)^{\eta-1}M_{\eta,\eta}(t-r)[G(z(r)+Cv(r)]dr \nonumber\\
 &\quad+\int_0^t(t-r)^{\eta-1}M_{\eta,\eta}(t-r)\hbar(r,z(r))dW(r),
\end{align}
where $M_{\eta}(t)$ and $M_{\eta,\eta}(t)$ denote the generalized Mittag-Leffler operators given by
\begin{align*}
M_{\eta}(t)=\int_0^\infty K_{\eta}(s)S_{\alpha}(t^\eta s)ds
\end{align*}
and
\begin{align*}
M_{\eta,\eta}(t)=\int_0^\infty \eta sK_{\eta}(s)S_{\alpha}(t^\eta s)ds,
\end{align*}
where $K_{\eta}:\mathbb{C}\rightarrow \mathbb{C}$ is the Mainardi function which is defined by
\begin{align*}
K_{\eta}(z) := \sum_{m=0}^\infty \frac{(-1)^m(z)^m}{m!\Gamma(-\eta m +1-\eta)}
\end{align*}
for each $\eta\in (0,1)$.
\end{definition}

Next, we recall some properties of $M_{\eta}(t)$ and $M_{\eta,\eta}(t)$, which are demonstrated  in \cite{zou}.
\begin{lemma}{\label{inq1}}\cite{zou}
The operators $M_{\eta}(t)$ and $M_{\eta,\eta}(t)$,  $t>0$, are linear bounded operators such that for $0\leq \beta <\alpha <2$, we have
\begin{align*}
\|M_{\eta}(t)z\|_{\mathcal{H}^{\beta}}\leq C_\alpha t^{-\frac{\eta\beta}{\alpha}}\|z\|,\,\,and \,\, \|M_{\eta,\eta}(t)z\|_{\mathcal{H}^{\beta}}\leq C_\eta t^{-\frac{\eta\beta}{\alpha}}\|z\|,
\end{align*}
where
\begin{align*}
C_\alpha= \frac{C_{\alpha,\beta}\Gamma(1-\frac{\beta}{\alpha})}{\Gamma(1-\frac{\eta\beta}{\alpha})}\,\, and\,\,C_\eta=\frac{\eta C_{\alpha,\beta}\Gamma(2-\frac{\beta}{\alpha})}{\Gamma(1+\eta(1-\frac{\beta}{\alpha}))}.
\end{align*}
\end{lemma}

\begin{lemma}{\label{inq2}}\cite{zou}
The operators $M_{\eta}(t)$ and $M_{\eta,\eta}(t)$, $t>0$, are strongly continuous and for any $0<T_0\leq \tau_1 <\tau_2 \leq T$ and $0\leq \beta <\alpha <2$,
we have
\begin{align*}
\|(M_{\eta}(\tau_2)-M_{\eta}(\tau_1))z\|_{\mathcal{H}^{\beta}}\leq C_{\alpha \beta} (\tau_2-\tau_1)^{\frac{\eta\beta}{\alpha}}\|z\|
\end{align*}
and
\begin{align*}
\|(M_{\eta,\eta}(\tau_2)-M_{\eta,\eta}(\tau_1))z\|_{\mathcal{H}^{\beta}}\leq C_{\eta \beta} (\tau_2-\tau_1)^{\frac{\eta\beta}{\alpha}}\|z\|,
\end{align*}
where
\begin{align*}
C_{\alpha \beta}= \frac{\alpha C_{\alpha,\beta}\Gamma(1-\frac{\beta}{\alpha})}{\beta T_0^{\frac{2\eta\beta}{\alpha}}\Gamma(1-\frac{\eta\beta}{\alpha})}\,\, and\,\,C_{\eta\beta}=\frac{\alpha\eta C_{\alpha,\beta}\Gamma(2-\frac{\beta}{\alpha})}{\beta T_0^{\frac{2\eta\beta}{\alpha}}\Gamma(1+\eta(1-\frac{\beta}{\alpha}))}.
\end{align*}
\end{lemma}
Now, we define the stochastic controllability operator
$L_T\in \mathcal{L}(U_{ad},L^p(\Sigma,H))$ by
\begin{align*}
L_Tv:=\int_0^T(T-s)^{\eta-1}M_{\eta,\eta}(T-s)Cv(s)ds
\end{align*}
and the corresponding adjoint operator $L_T^*:L^p(\Sigma,H)\rightarrow U_{ad}$ by
\begin{align*}
L_T^*z=C^*M_{\eta,\eta}^*(T-s)E\{z|\mu_t\}.
\end{align*}
Similarly to the Grammian matrix, we have the stochastic Grammian operator
\begin{align*}
\Upsilon_0^T:=L_T(L_T^*)z=\int_0^T(T-s)^{\eta-1}M_{\eta,\eta}(T-s)CC^*M^*_{\eta,\eta}(T-s)E\{z|\mu_s\}ds.
\end{align*}

Next, we define the reachable set $K(T):=\{z(T,z_0,v):v(\cdot)\in U_{ad}\}$ of $(\ref{BE2})$, which is the set of all final states $z$ with initial state $z_0$ and control $v$ at the terminal time $T$ .

\begin{definition} The time-fractional SNSE $(\ref{BE2})$ is called approximately controllable on $[0,T]$ if $\overline{K(T)}=L^p(\Sigma,H)$.
\end{definition}

\begin{lemma}\cite{Dauer}
For any $z_T\in L^p(\Sigma,H)$, there exists $\varphi\in L^p_\mu(\Sigma,L^p((0,T),L^p_0))$ such that $z_T=Ez_T+\int_0^T\varphi(s)dW(s)$.
\end{lemma}

\section{Existence and Controllability Results}
Assume the following conditions:
\begin{enumerate}
\item [(1)] The bounded bilinear operator $G:H\rightarrow H^{-1}(\mathcal{O})$ satisfies the conditions
\begin{align*}
\|G(z)\|_{H^{-1}}\leq C_1\|z\|^2
\end{align*}
and
\begin{align*}
\|G(z)-G(w)\|_{H^{-1}}\leq C_2(\|z\|+\|w\|)\|z-w\|,
\end{align*}
where $C_1$ and $C_2$ are positive constants.
\item [(2)] $\hbar:[0,T]\times H\rightarrow L^2_0$ is a measurable function which satisfies the conditions
     \begin{align*}
\|\hbar(t,z(t))\|_{L_0^2}\leq L_1(1+\|z\|)
\end{align*}
and
\begin{align*}
\|\hbar(t,z(t))-\hbar(t,w(t))\|_{L_0^2}\leq L_2(\|z-w\|),
\end{align*}
where $L_1$ and $L_2$ are positive constants.
\item [(3)] The linear deterministic system corresponding to (\ref{BE2}),
\begin{eqnarray}\label{BE3}
\begin{split}
& ^CD^{\eta} z(t)=-A_{\alpha} z(t)+G(z(t))+Cv(t), \quad t\in(0,T],\\
&z(0)=z_0,\\
\end{split}
\end{eqnarray}
is approximately controllable on $[t,T]$, that is, for each $t\in [0,T)$, the operator $\lambda(\lambda I+\Upsilon_0^T)\rightarrow 0$ as $\lambda\rightarrow 0+$ in the strong operator topology.
\end{enumerate}

For any $z_T\in L^p(\Sigma, H^\beta)$, we define the control function
\begin{align*}
v^\lambda(t,z)=&C^*M^*_{\eta,\eta}(T-t)\bigg[(\lambda I+\Upsilon_0^T)^{-1}(Ez_T-M_\eta(T)z_0)+\int_0^t(\lambda I+\Upsilon_0^T)^{-1}\varphi(r)dW(r)\\
&\,-\int_0^t(\lambda I+\Upsilon_0^T)^{-1}M_{\eta,\eta}(T-r)G(z(r))dr-\int_0^t(\lambda I+\Upsilon_0^T)^{-1}M_{\eta,\eta}(T-r)\hbar(r,z(r))dW(r)\bigg].
\end{align*}

\begin{lemma}\label{vlemma}
For $p\geq 2$, $0\leq \beta <\alpha <2$, and for all $z,w\in L^p(\Sigma, H^\beta)$, we have
\begin{align}
&E\|v^\lambda(t,z)-v^\lambda(t,w)\|^p\leq \frac{C_v}{\lambda^p}\int_0^tE\|z(r)-w(r)\|^pdr,\label{v1}\\
&E\|v^\lambda(t,z)\|^p\leq \frac{C_v}{\lambda^p}\bigg(1+\int_0^tE\|z(r)\|^pdr\bigg),\label{v2}
\end{align}
where $C_v$ denotes a constant.
\end{lemma}
\begin{proof}
For $z,w\in L^p(\Sigma, H^\beta)$, using the bounds of linear operator $M_{\eta,\eta}$ from Lemma \ref{inq1}, H\"{o}lder's inequality and conditions (1)--(2), we obtain
\begin{align}
E\|&v^\lambda(t,z)-v^\lambda(t,w)\|^p\nonumber\\
&\leq 2^{p-1}\bigg[E\bigg\|C^*M^*_{\eta,\eta}(T-t)\int_0^t(\lambda I+\Upsilon_0^T)^{-1}M_{\eta,\eta}(T-r)[G(z(r))-G(w(r))]dr\bigg\|^p\nonumber\\
&\quad\quad+E\bigg\|C^*M^*_{\eta,\eta}(T-t)\int_0^t(\lambda I+\Upsilon_0^T)^{-1}M_{\eta,\eta}(T-r)[\hbar(r,z(r))-\hbar(r,w(r))]dW(r)\bigg\|^p\bigg]\nonumber\\
&\leq 2^{p-1}\|C^*\|^p\|M^*_{\eta,\eta}(T-t)\|^p\bigg[E\bigg\|\int_0^t(\lambda I+\Upsilon_0^T)^{-1}A_1M_{\eta,\eta}(T-r)A_{-1}[G(z(r))-G(w(r))]dr\bigg\|^p\nonumber\\
&\quad\quad\quad+E\bigg\|\int_0^t(\lambda I+\Upsilon_0^T)^{-1}M_{\eta,\eta}(T-r)[\hbar(r,z(r))-\hbar(r,w(r))]dW(r)\bigg\|^p\bigg]\nonumber\\
&\leq 2^{p-1}\frac{\|C^*\|^p}{\lambda^p}C_\eta^p\bigg[E\bigg\|\int_0^tA_1M_{\eta,\eta}(T-r)A_{-1}[G(z(r))-G(w(r))]dr\bigg\|^p\nonumber\\
&\quad\quad\quad+E\bigg\|\int_0^tM_{\eta,\eta}(T-c)[\hbar(r,z(r))-\hbar(r,w(r))]dW(r)\bigg\|^p\bigg]\nonumber\\
&\leq 2^{p-1}\frac{\|C^*\|^p}{\lambda^p}C_\eta^p[I_1+I_2],\label{L1}
\end{align}
where
\begin{equation*}
I_1=E\bigg\|\int_0^tA_1M_{\eta,\eta}(T-r)A_{-1}[G(z(r))-G(w(r))]dr\bigg\|^p  
\end{equation*}
and
\begin{equation*}
I_2=E\bigg\|\int_0^tM_{\eta,\eta}(T-r)[\hbar(r,z(r))-\hbar(r,w(r))]dW(r)\bigg\|^p.  
\end{equation*}
Using H\"{o}lder's inequality, we obtain
\begin{align}    
I_1&=E\bigg\|\int_0^tA_1M_{\eta,\eta}(T-r)A_{-1}[G(z(r))-G(w(r))]dr\bigg\|^p\nonumber\\
&\quad \leq \bigg(\int_0^t\|M_{\eta,\eta}(T-r)\|^{\frac{p}{p-1}}_{H^1}dr\bigg)^{p-1}\bigg(\int_0^tE\|[G(z(r))-G(w(r))]\|^p_{H^{-1}}dr\bigg)\nonumber\\
&\quad \leq \bigg(\int_0^t C_\eta^{\frac{p}{p-1}}(T-r)^{-\frac{\eta p}{\alpha(p-1)}}dr\bigg)^{p-1}\bigg(\int_0^tC_2^p(\max_{t\in[0,T]}E[\|z(t)\|^p]+\max_{t\in[0,T]}E[\|w(t)\|^p])\nonumber\\
&\quad\quad\quad\quad\quad\quad\quad \times E\|z(r)-w(r)\|^pdr\bigg)\nonumber\\
&\quad \leq C_\eta^pC_2^p\bigg(\max_{t\in[0,T]}E[\|z(t)\|^p]+\max_{t\in[0,T]}E[\|w(t)\|^p]\bigg)\bigg(\int_0^t(T-r)^{-\frac{\eta p}{\alpha(p-1)}}dr\bigg)^{p-1}\nonumber\\
&\quad\quad\quad\quad\quad\quad\quad\times\int_0^tE\|z(r)-w(r)\|^pdr\nonumber\\
&\quad \leq C_\eta^pC_2^p\bigg(\max_{t\in[0,T]}E[\|z(t)\|^p]+\max_{t\in[0,T]}E[\|w(t)\|^p]\bigg)T^{p[
1-\frac{\eta}{\alpha}]-1}\bigg[\frac{p-1}{p[
1-\frac{\eta}{\alpha}]-1}\bigg]^{p-1}\nonumber\\
&\quad\quad\quad\quad\quad\quad\quad\times\int_0^tE\|z(r)-w(r)\|^pdr,\label{L2}
\end{align}
and using the Burkholder-Davis-Gundy inequality from Lemma \ref{Gundys},
\begin{align}
I_2&=E\bigg\|\int_0^tM_{\eta,\eta}(T-r)[\hbar(r,z(r))-\hbar(r,w(r))]dW(r)\bigg\|^p\nonumber\\
&\quad\quad\leq \kappa(p)E\bigg[\bigg(\int_0^t\|M_{\eta,\eta}(T-r)[\hbar(r,z(r))-\hbar(r,w(r))]\|_{L^2_0}^2dr\bigg)^{\frac{p}{2}}\bigg]\nonumber\\
&\quad\quad\leq \kappa(p)\bigg(\int_0^t\|M_{\eta,\eta}(T-r)\|^{\frac{2p}{p-2}}dr\bigg)^{\frac{p-2}{2}}\bigg(\int_0^tE\|[\hbar(r,z(r))-\hbar(r,w(r))]\|_{L^2_0}^pdr\bigg)\nonumber\\
&\quad\quad\leq \kappa(p)C_\eta^pL_2^p\int_0^tE\|z(r)-w(r)\|^pdr.\label{L3}
\end{align}
Using (\ref{L2}) and (\ref{L3}) in (\ref{L1}), we infer that
\begin{align*}
E\|&v^\lambda(t,z)-v^\lambda(t,w)\|^p\leq \frac{C_v}{\lambda^p}\int_0^tE\|z(r)-w(r)\|^pdr,
\end{align*}
where
\begin{align*}
C_v=&2^{p-1}\|C^*\|^pC_\eta^{2p}\bigg\{C_2^p\bigg(\max_{t\in[0,T]}E[\|z(t)\|^p]\\
&\quad\quad+\max_{t\in[0,T]}E[\|w(t)\|^p]\bigg)T^{p[1-
\frac{\eta}{\alpha}]-1}\bigg[\frac{p-1}{p[
1-\frac{\eta}{\alpha}]-1}\bigg]^{p-1}+ \kappa(p)L_2^p\bigg\}.
\end{align*}
Since inequality (\ref{v2}) can be obtained in a similar manner, we omit its proof here.
\end{proof}

For any $\lambda>0$, we define the operator $\mathcal{F}_\lambda (z(t)):L^p(\Sigma, H^\beta)\rightarrow L^p(\Sigma, H^\beta)$ by
\begin{align*}
\mathcal{F}_\lambda (z(t)) := &M_{\eta}(t)z_0+\int_0^t(t-r)^{\eta-1}M_{\eta,\eta}(t-r)[G(z(r))+Cv^\lambda(r,z)]dr\nonumber\\
 &\quad+\int_0^t(t-r)^{\eta-1}M_{\eta,\eta}(t-r)\hbar(r,z(r))dW(r).
\end{align*}

\begin{theorem}\label{MainThm}
If conditions $(1)$ and $(2)$ are satisfied, then the time-fractional SNSE (\ref{BE2}) has a unique mild solution $(z(t))_{t\in[0,T]}$ in $L^p(\Sigma, H^\beta)$ for $p\geq 2$, $\eta p\neq 1$ and $0\leq \beta <\alpha <2$.
\end{theorem}

To prove this result, we use the Banach contraction principle to demonstrate that the operator $\mathcal{F}_\lambda (z(t))$ has a fixed point, which is a mild solution to (\ref{BE2}). To this end, first we prove the following Lemma

\begin{lemma}
 For $p\geq 2$, $0\leq \beta <\alpha <2$, and for any $z\in L^p(\Sigma, H^\beta)$, the operator $\mathcal{F}_\lambda (z(t))$ is continuous on $[0,T]$ in the $L^p$ sense.
\end{lemma}
\begin{proof}
For $0\leq \tau_1<\tau_2\leq T$ and a fixed $z\in L^p(\Sigma, H^\beta)$, we have
\begin{align}
E\|&\mathcal{F}_\lambda (z(\tau_2))-\mathcal{F}_\lambda (z(\tau_1))\|^p_{H^\beta}\nonumber\\
&\leq 4^{p-1}\bigg[E\|(M_{\eta}(\tau_2)-M_{\eta}(\tau_1))z_0\|^p_{H^\beta}\nonumber\\
&\quad +E\|\int_0^{\tau_2}(\tau_2-r)^{\eta-1}M_{\eta,\eta}(\tau_2-r)G(z(r))dr-\int_0^{\tau_1}(\tau_1-r)^{\eta-1}M_{\eta,\eta}(\tau_1-r)G(z(r))dr\|^p_{H^\beta}\nonumber\\
&\quad +E\|\int_0^{\tau_2}(\tau_2-r)^{\eta-1}M_{\eta,\eta}(\tau_2-r)Cv^\lambda(r,z)dr-\int_0^{\tau_1}(\tau_1-r)^{\eta-1}M_{\eta,\eta}(\tau_1-r)Cv^\lambda(r,z)dr\|^p_{H^\beta}\nonumber\\
 &\quad+E\|\int_0^{\tau_2}(\tau_2-r)^{\eta-1}M_{\eta,\eta}(\tau_2-r)\hbar(r,z(r))dW(r)\nonumber\\
 &\quad\quad\quad\quad-\int_0^{\tau_1}(\tau_1-r)^{\eta-1}M_{\eta,\eta}(\tau_1-r)\hbar(r,z(r))dW(r)\|^p_{H^\beta}\bigg]\nonumber\\
 &\leq 4^{p-1}\bigg[\sum_{j=1}^4 J_j\bigg],\label{cts}
\end{align}
where
\begin{align*}
J_1=E[\|(M_{\eta}(\tau_2)-M_{\eta}(\tau_1))z_0\|^p_{H^\beta}],
\end{align*}
\begin{align*}
J_2&=E\|\int_0^{\tau_2}(\tau_2-r)^{\eta-1}M_{\eta,\eta}(\tau_2-r)G(z(r))dr-\int_0^{\tau_1}(\tau_1-r)^{\eta-1}M_{\eta,\eta}(\tau_1-r)G(z(r))dr\|^p_{H^\beta}\\
&\leq 3^{p-1}\bigg[E[\|\int_0^{\tau_1}(\tau_1-r)^{\eta-1}[M_{\eta,\eta}(\tau_2-r)-M_{\eta,\eta}(\tau_1-r)]G(z(r))dr\|^p_{H^\beta}\\
&\quad +E[\|\int_0^{\tau_1}[(\tau_2-r)^{\eta-1}-(\tau_1-r)^{\eta-1}]M_{\eta,\eta}(\tau_2-r)G(z(r))dr\|^p_{H^\beta}]\\
&\quad +E[\|\int_{\tau_1}^{\tau_2}(\tau_2-r)^{\eta-1}M_{\eta,\eta}(\tau_2-r)G(z(r))dr\|^p_{H^\beta}\bigg]\\
&=3^{p-1}(J_{21}+J_{22}+J_{23}),
\end{align*}
\begin{align*}
J_{3}&=E\|\int_0^{\tau_2}(\tau_2-r)^{\eta-1}M_{\eta,\eta}(\tau_2-r)Cv^\lambda(r,z)]dr-\int_0^{\tau_1}(\tau_1-c)^{\eta-1}M_{\eta,\eta}(\tau_1-r)Cv^\lambda(r,z)]dr\|^p_{H^\beta}\\
&\leq 3^{p-1}\bigg[E[\|\int_0^{\tau_1}(\tau_1-r)^{\eta-1}[M_{\eta,\eta}(\tau_2-r)-M_{\eta,\eta}(\tau_1-r)]Cv^\lambda(r,z)dr\|^p_{H^\beta}\\
&\quad +E[\|\int_0^{\tau_1}[(\tau_2-r)^{\eta-1}-(\tau_1-r)^{\eta-1}]M_{\eta,\eta}(\tau_2-r)Cv^\lambda(r,z)dr\|^p_{H^\beta}]\\
&\quad +E[\|\int_{\tau_1}^{\tau_2}(\tau_2-r)^{\eta-1}M_{\eta,\eta}(\tau_2-r)Cv^\lambda(r,z)dr\|^p_{H^\beta}\bigg]\\
&=3^{p-1}(J_{31}+J_{32}+J_{33}),
\end{align*}
and
\begin{align*}
J_{4}&=E\|\int_0^{\tau_2}(\tau_2-r)^{\eta-1}M_{\eta,\eta}(\tau_2-r)\hbar(r,z(r))dW(r)\\
&\quad\quad\quad\quad-\int_0^{\tau_1}(\tau_1-r)^{\eta-1}M_{\eta,\eta}(\tau_1-r)\hbar(r,z(r))dW(r)\|^p_{H^\beta}\\
&\leq 3^{p-1}\bigg[E[\|\int_0^{\tau_1}(\tau_1-r)^{\eta-1}[M_{\eta,\eta}(\tau_2-r)-M_{\eta,\eta}(\tau_1-r)]\hbar(r,z(r))dW(r)\|^p_{H^\beta}\\
&\quad +E[\|\int_0^{\tau_1}[(\tau_2-r)^{\eta-1}-(\tau_1-r)^{\eta-1}]M_{\eta,\eta}(\tau_2-r)\hbar(r,z(r))dW(r)\|^p_{H^\beta}]\\
&\quad +E[\|\int_{\tau_1}^{\tau_2}(\tau_2-r)^{\eta-1}M_{\eta,\eta}(\tau_2-r)\hbar(r,z(r))dW(r)\|^p_{H^\beta}\bigg]\\
&=3^{p-1}(J_{41}+J_{42}+J_{43}).
\end{align*}
Using Lemma \ref{inq2}, we see that
\begin{align*}
J_1\leq C_{\alpha\beta}^p(\tau_2-\tau_1)^{\frac{p\eta\beta}{\alpha}}E[\|z_0\|^p].
\end{align*}
Next, using Lemma \ref{inq1}, Lemma \ref{inq2}, condition (1) and H\"{o}lder's inequality, we find that
\begin{align*}
J_{21}&=E[\|\int_0^{\tau_1}(\tau_1-r)^{\eta-1}A_\beta[M_{\eta,\eta}(\tau_2-r)-M_{\eta,\eta}(\tau_1-r)]G(z(r))dr\|^p]\\
&\quad\leq C_{\eta\beta}^p(\tau_2-\tau_1)^{\frac{p\eta(\beta+1)}{\alpha}}\bigg(\int_0^{\tau_1}(\tau_1-r)^{\frac{(\eta-1)p}{p-1}}dr\bigg)^{p-1}\bigg(\int_0^{\tau_1}E\|G(z(r))\|^p_{H^{-1}}dr\bigg)\\
&\quad\leq C_{\eta\beta}^pC_1^pT^{\eta p}\bigg[\frac{p-1}{\eta p-1}\bigg]^{p-1}\sup_{t\in[0,T]}E[\|z(r)\|^{2p}](\tau_2-\tau_1)^{\frac{p\eta(\beta+1)}{\alpha}},
\end{align*}
\begin{align*}
J_{22}&=E[\|\int_0^{\tau_1}[(\tau_2-r)^{\eta-1}-(\tau_1-r)^{\eta-1}]A_\beta M_{\eta,\eta}(\tau_2-r)G(z(r))dr\|^p]\\
&\leq C_\eta^p\bigg(\int_0^{\tau_1}\bigg\{[(\tau_2-r)^{\eta-1}-(\tau_1-r)^{\eta-1}](\tau_2-r)^{\frac{-(\beta+1)\eta}{\alpha}}\bigg\}^{\frac{p}{p-1}}dr\bigg)^{p-1}\bigg(\int_0^{\tau_1}E[\|G(z(r))\|^p_{H^{-1}}]dr\bigg)\\
&\leq 2C_\eta^pC_1^pT\bigg\{\frac{p-1}{p[\eta-\frac{\eta(\beta+1)}{\alpha}]-1}\bigg\}^{p-1}\sup_{t\in[0,T]}E[\|z(r)\|^{2p}](\tau_2-\tau_1)^{\frac{p\eta(\alpha-\beta-1)-\alpha}{\alpha}},
\end{align*}
and
\begin{align*}
J_{23}&=E[\|\int_{\tau_1}^{\tau_2}(\tau_2-r)^{\eta-1}M_{\eta,\eta}(\tau_2-r)G(z(r))dr\|^p_{H^\beta}\bigg]\\
&\leq C_\eta^p\bigg(\int_{\tau_1}^{\tau_2}[(\tau_2-r)^{\eta-1-\frac{\eta(\beta+1)}{\alpha}}]^{\frac{p}{p-1}}dr\bigg)^{p-1}\int_{\tau_1}^{\tau_2}E[\|G(z(r))\|^p_{H^{-1}}]dr\\
&\leq C_\eta^pC_1^p\bigg\{\frac{p-1}{p[\eta-\frac{\eta(\beta+1)}{\alpha}]-1}\bigg\}^{p-1}\sup_{t\in[0,T]}E[\|z(r)\|^{2p}](\tau_2-\tau_1)^{\frac{p\eta}{\alpha}(\alpha-\beta-1)}.
\end{align*}
Next, using Lemma \ref{inq1}, Lemma \ref{inq2}, and Lemma \ref{vlemma}, we obtain
\begin{align*}
J_{31}&=\bigg[E[\|\int_0^{r_1}(\tau_1-r)^{\eta-1}[M_{\eta,\eta}(\tau_2-r)-M_{\eta,\eta}(\tau_1-r)]Cv^\lambda(r,z)dr\|^p_{H^\beta}\bigg]\\
&\leq C_{\eta\beta}^p(\tau_2-\tau_1)^{\frac{p\eta\beta}{\alpha}}(\int_0^{\tau_1}(\tau_1-r)^{\frac{(\eta-1)p}{p-1}}dr)^{p-1}\bigg(\int_0^{\tau_1}E\|Cv^\lambda(r,z)\|^pdr\bigg)\\
&\quad\leq C_{\eta\beta}^p\|C\|^p\frac{C_v}{\lambda^p}T^{\eta p}\bigg[\frac{p-1}{\eta p-1}\bigg]^{p-1}[1+\int_0^tE[\|z(r)\|^{p}dr](\tau_2-\tau_1)^{\frac{p\eta\beta}{\alpha}},
\end{align*}
\begin{align*}
J_{32}&=E[\|\int_0^{\tau_1}[(\tau_2-r)^{\eta-1}-(\tau_1-r)^{\eta-1}]M_{\eta,\eta}(\tau_2-r)Cv^\lambda(r,z)dr\|^p_{H^\beta}]\\
&\leq C_\eta^p\bigg(\int_0^{\tau_1}\bigg\{[(\tau_2-r)^{\eta-1}-(\tau_1-r)^{\eta-1}](\tau_2-r)^{\frac{-(\beta+1)\eta}{\alpha}}\bigg\}^{\frac{p}{p-1}}dr\bigg)^{p-1}\bigg(\int_0^{\tau_1}E[\|Cv^\lambda(r,z)\|^p]dr\bigg)\\
&\leq C_\eta^p\|C\|^p\frac{C_v}{\lambda^p}T\bigg\{\frac{p-1}{p[\eta-\frac{\eta(\beta+1)}{\alpha}]-1}\bigg\}^{p-1}[1+\int_0^tE[\|z(r)\|^{p}dr](\tau_2-\tau_1)^{\frac{p\eta(\alpha-\beta-1)-\alpha}{\alpha}},
\end{align*}
and
\begin{align*}
J_{33}&=E[\|\int_{\tau_1}^{\tau_2}(\tau_2-r)^{\eta-1}M_{\eta,\eta}(\tau_2-r)Cv^\lambda(r,z)dr\|^p_{H^\beta}]\\
&\leq C_\eta^p\bigg(\int_{\tau_1}^{\tau_2}[(\tau_2-r)^{\eta-1-\frac{\eta(\beta+1)}{\alpha}}]^{\frac{p}{p-1}}dr\bigg)^{p-1}\int_{\tau_1}^{\tau_2}E[\|Cv^\lambda(r,z)\|^p]dr\\
&\leq C_\eta^p\|C\|^p\frac{C_v}{\lambda^p}\bigg\{\frac{p-1}{p[\eta-\frac{\eta(\beta+1)}{\alpha}]-1}\bigg\}^{p-1}[1+\int_0^tE[\|z(r)\|^{p}dc](\tau_2-\tau_1)^{\frac{p\eta}{\alpha}(\alpha-\beta-1)}.
\end{align*}
Using Lemma \ref{Gundys} and condition $(2)$, we get
\begin{align*}
J_{41}&=E\bigg[\|\int_0^{\tau_1}(\tau_1-r)^{\eta-1}[M_{\eta,\eta}(\tau_2-r)-M_{\eta,\eta}(\tau_1-r)]\hbar(r,z(r))dW(r)\|^p_{H^\beta}\bigg]\\
&\leq \kappa(p)E\bigg[\bigg(\int_0^{\tau_1}\|(\tau_1-r)^{\eta-1}A_\beta[M_{\eta,\eta}(\tau_2-r)-M_{\eta,\eta}(\tau_1-r)]\|^2\|\hbar(r,z(r))\|_{L^2_0}^2dr\bigg)^{\frac{p}{2}}\bigg]\\
&\leq \kappa(p)C_{\eta\beta}^p(\tau_2-\tau_1)^{\frac{p\eta\beta}{\alpha}}\bigg(\int_0^{\tau_1}(\tau_1-r)^{\frac{2p(\eta-1)}{p-2}}dr\bigg)^{\frac{p-2}{2}}\int_0^{\tau_1}E\|\hbar(r,z(r))\|^p_{L^2_0}dr\\
&\leq \kappa(p)C_{\eta\beta}^pL_1^pT^{\frac{2p\eta-p-1}{2}}\bigg[\frac{p-1}{2p\eta-p-2}\bigg]^{p-1}[1+\sup_{t\in[0,T]}E[\|z(t)\|^p]](\tau_2-\tau_1)^{\frac{p\eta\beta}{\alpha}},
\end{align*}
\begin{align*}
J_{42}&=E\bigg[\bigg\|\int_0^{\tau_1}[(\tau_2-r)^{\eta-1}-(\tau_1-r)^{\eta-1}]M_{\eta,\eta}(\tau_2-r)\hbar(r,z(r))dW(r)\bigg\|^p_{H^\beta}\bigg]\\
&\leq \kappa(p)E\bigg[\bigg(\int_0^{\tau_1}\|[(\tau_2-r)^{\eta-1}-(\tau_1-r)^{\eta-1}]A_\beta M_{\eta,\eta}(\tau_2-r)\|^2\|\hbar(r,z(r))\|^2_{L^2_0}dr\bigg)^{\frac{p}{2}}\\
&\leq \kappa(p)C_\eta^p\bigg(\int_0^{\tau_1}\{[(\tau_2-r)^{\eta-1}-(\tau_1-r)^{\eta-1}](\tau_2-r)^{\frac{-\eta\beta}{\alpha}}\}^{\frac{2p}{p-2}}dr\bigg)^{\frac{p-2}{2}}\int_0^{\tau_1}E\|\hbar(r,z(r))\|^p_{L^2_0}dr\\
&\leq \kappa(p)C_\eta^pL_1^pT\bigg[\frac{\alpha(p-2)}{2p\eta(\alpha-\beta)-(p+2)\alpha}\bigg]^{\frac{p-2}{2}}[1+\sup_{t\in[0,T]}E\|z(t)\|^p](\tau_2-\tau_1)^{\frac{2p\eta(\alpha-\beta)-(p+2)\alpha}{2\alpha}},
\end{align*}
and
\begin{align*}
J_{43}&=E\bigg[\bigg\|\int_{\tau_1}^{\tau_2}(\tau_2-r)^{\eta-1}M_{\eta,\eta}(\tau_2-r)\hbar(r,z(r))dW(r)\bigg\|^p_{H^\beta}\bigg]\\
&\leq \kappa(p)E\bigg[\bigg(\int_{\tau_1}^{\tau_2}\|(\tau_2-r)^{\eta-1}A_\beta M_{\eta,\eta}(\tau_2-r)\|^2\|\hbar(r,z(r))\|^2_{L^2_0}dr\bigg)^{\frac{p}{2}}\bigg]\\
&\leq \kappa(p)C_\eta^p\bigg(\int_{\tau_1}^{\tau_2}[(\tau_2-r)^{\eta-1-\frac{\eta\beta}{\alpha}}]^{\frac{2p}{p-2}}dr\bigg)^{\frac{p-2}{2}}\int_{\tau_1}^{\tau_2}E\|\hbar(r,z(r))\|^p_{L^2_0}dr\\
&\leq \kappa(p)C_\eta^pL_1^p\bigg[\frac{\alpha(p-2)}{2p\eta(\alpha-\beta)-(p+2)\alpha}\bigg]^{\frac{p-2}{2}}[1+\sup_{t\in[0,T]}E\|z(t)\|^p](\tau_2-\tau_1)^{\frac{2p\eta(\alpha-\beta)-p\alpha}{2\alpha}}.
\end{align*}
Combining the above inequalities and plugging them into (\ref{cts}), we see that 
$$E\|\mathcal{F}_\lambda (z(\tau_2))-\mathcal{F}_\lambda (z(\tau_1))\|^p_{H^\beta}\rightarrow 0 \;\text{as}\; \tau_2-\tau_1\rightarrow 0.$$ Hence $\mathcal{F}_\lambda(z(t))$ is continuous on $[0,T]$.\\

\textbf{Proof of Theorem \ref{MainThm}}
We first show that $\mathcal{F}_\lambda$ maps $L^p(\Sigma, H^\beta)$ into $L^p(\Sigma, H^\beta)$. Indeed, for $z\in L^p(\Sigma, H^\beta)$, we have
\begin{align*}
E\|\mathcal{F}_\lambda(z(t))\|^p_{H^\beta}\leq 4^{p-1}\bigg[\sup_{t\in [0,T]}E\|M_\eta(t)z_0\|^p_{H^\beta}+\sup_{t\in [0,T]}\sum_{i=1}^3E\|\Theta_i^z(t)\|^p_{H^\beta}\bigg],
\end{align*}
where
\begin{align*}
\sup_{t\in [0,T]}E\|M_\eta(t)z_0\|^p_{H^\beta}\leq C_\alpha^pT^{\frac{-p\eta\beta}{\alpha}}E\|z_0\|^p,
\end{align*}
\begin{align*}
\sup_{t\in [0,T]}E\|\Theta_1^z(t)\|^p_{H^\beta}&\leq \sup_{t\in [0,T]}E\|\int_0^t(t-r)^{\eta-1}M_{\eta,\eta}(t-r)Gz(r)dr\|^p_{H^\beta}\\
&\leq \sup_{t\in [0,T]}E\|\int_0^t(t-r)^{\eta-1}A_{\beta+1}M_{\eta,\eta}(t-r)A_{-1}Gz(r)dr\|^p\\
&\leq \sup_{t\in [0,T]}C_\eta^p\bigg(\int_0^t(t-r)^{\frac{p[\eta-1-\frac{\eta(\beta+1)}{\alpha}]}{p-1}}dr\bigg)^{p-1}\int_0^tE[\|A_{-1}G(z(r))\|^p]dr\\
&\leq C_\eta^pC_1^p\bigg\{\frac{p-1}{p[\eta-\frac{\eta(\beta+1)}{\alpha}]-1}\bigg\}^{p-1}T^{p[\eta-\frac{\eta(\beta+1)}{\alpha}]-1}\max_{t\in[0,T]}E\|z(t)\|^2,
\end{align*}
\begin{align*}
\sup_{t\in [0,T]}E\|\Theta_2^z(t)\|^p_{H^\beta}&\leq \sup_{t\in [0,T]}E\|\int_0^t(t-r)^{\eta-1}M_{\eta,\eta}(t-r)Cv^\lambda (r,z(r))dr\|^p_{H^\beta}\\
&\leq \sup_{t\in [0,T]}E\|\int_0^t(t-r)^{\eta-1}A_{\beta}M_{\eta,\eta}(t-r)Cv^\lambda (r,z(r))dr\|^p\\
&\leq \sup_{t\in [0,T]}C_\eta^p\bigg(\int_0^t(t-r)^{\frac{p[\eta-1-\frac{\eta\beta}{\alpha}]}{p-1}}dr\bigg)^{p-1}\int_0^tE[\|Cv^\lambda (r,z(r))\|^p]dr\\
&\leq C_\eta^p\|C\|^p\frac{C_v}{\lambda^p}\bigg\{\frac{p-1}{p[\eta-\frac{\eta\beta}{\alpha}]-1}\bigg\}^{p-1}T^{p[\eta-\frac{\eta\beta}{\alpha}]}\bigg[1+\max_{t\in[0,T]}\int_0^tE\|z(r)\|^pdr\bigg],
\end{align*}
and
\begin{align*}
\sup_{t\in [0,T]}E\|\Theta_3^z(t)\|^p_{H^\beta}&\leq \sup_{t\in [0,T]}E\|\int_0^t(t-r)^{\eta-1}M_{\eta,\eta}(t-r)\hbar(r,z(r))dW(r)\|^p_{H^\beta}\\
&\leq \kappa(p)\sup_{t\in [0,T]}E\bigg[\bigg(\int_0^t\|(t-r)^{\eta-1}A_{\beta}M_{\eta,\eta}(t-r)\|^2\|\hbar(r,z(r))\|^2_{L^{2}_0}dr\bigg)^{\frac{p}{2}}\bigg]\\
&\leq \kappa(p)C_\eta^p\bigg(\int_0^t(t-r)^{\frac{2p[\eta-1-\frac{\eta\beta}{\alpha}]}{p-2}}dr\bigg)^{\frac{p-2}{2}}\int_0^tE\|\hbar(r,z(r))\|^p_{L^2_0}dr\\
&\leq \kappa(p)C_\eta^pL_1^p\bigg[\frac{p-2}{p(2\beta-1-\frac{\eta\beta}{\alpha})-2}\bigg]^{\frac{p-2}{2}}T^{\frac{p(2\beta-1-\frac{\eta\beta}{\alpha})-2}{2}}[1+\sup_{t\in[0,T]}E\|z(t)\|^p].
\end{align*}
Hence $E\|\mathcal{F}_\lambda(z(t))\|^p_{H^\beta}\leq \infty$. Using the continuity of the operator $\mathcal{F}_\lambda$, we get that $\mathcal{F}_\lambda z\in H^\beta$.
Hence for each $\lambda>0$, $\mathcal{F}_\lambda$ indeed maps $L^p(\Sigma, H^\beta)$ into itself, as asserted.
Next, for $z,w\in L^p(\Sigma, H^\beta)$, we have
\begin{align*}
E\|(\mathcal{F}_\lambda z)(t)-(\mathcal{F}_\lambda w)(t)\|^p_{H^\beta}&\leq 3^{p-1}\bigg\{E\|\int_0^t(t-r)^{\eta-1}M_{\eta,\eta}(t-r)[Gz(r)-Gw(r)]dr\|^p_{H^\beta}\\
&\quad+E\|\int_0^t(t-r)^{\eta-1}M_{\eta,\eta}(t-r)[Cv^{\lambda}(r,z(r))-Cv^{\lambda}(r,w(r))]dr\|^p_{H^\beta}\\
&\quad +E\|\int_0^t(t-r)^{\eta-1}M_{\eta,\eta}(t-r)[\hbar(r,z(r)-\hbar(r,w(r)))]dW(r)\|^p_{H^{\beta}}\bigg\}\\
&= J_1+J_2+J_3,
\end{align*}
where
\begin{align*}
J_1&=E\|\int_0^t(t-r)^{\eta-1}M_{\eta,\eta}(t-r)[Gz(r)-Gw(r)]dr\|^p_{H^{\beta}}\\
&\leq E\|\int_0^t(t-r)^{\eta-1}A_1M_{\eta,\eta}(t-r)A_{\beta-1}[Gz(r)-Gw(r)]dr\|^p\\
&\leq C_\eta^p E\bigg(\int_0^t(t-r)^{(\eta-1-\frac{\eta}{\alpha})\frac{p}{p-1}}dr\bigg)^{p-1}\bigg(\int_0^tE\|A_{\beta}[Gz(r)-Gw(r)]dr\|^p_{H^{-1}}dr\bigg)\\
&\leq C_\eta^p C_2^p\bigg(\max_{t\in[0,T]}E[\|z(t)\|^p_{H^\beta}]+\max_{t\in[0,T]}E[\|w(t)\|^p_{H^\beta}]\bigg) \bigg[\frac{T^{(\eta-1-\frac{\eta}{\alpha})\frac{p}{p-1}}}{(\eta-1-\frac{\eta}{\alpha})\frac{p}{p-1}}\bigg]^{p-1}\bigg[\int_0^tE\|z(r)-w(r)\|^p_{H^\beta}dr\bigg]\\
&\leq C_\eta^p C_2^p\bigg(\max_{t\in[0,T]}E[\|z(t)\|^p_{H^\beta}]+\max_{t\in[0,T]}E[\|w(t)\|^p_{H^\beta}]\bigg)\bigg[\frac{p-1}{p\eta[1-\frac{1}{\alpha}]-1}\bigg]^{p-1}T^{p\eta[1-\frac{1}{\alpha}]-1}\\
&\quad\quad\quad\times\int_0^tE\|z(r)-w(r)\|^p_{H^\beta}dr,
\end{align*}
\begin{align*}
J_2&=E\|\int_0^t(t-r)^{\eta-1}M_{\eta,\eta}(t-r)[Cv^\lambda(r,z(r))-Cv^\lambda(r,w(r))]dr\|^p_{H^{\beta}}\\
&\leq E\|\int_0^t(t-r)^{\eta-1}M_{\eta,\eta}(t-r)A_{\beta}[Cv^\lambda(r,z(r))-Cv^\lambda(r,w(r))]dr\|^p\\
&\leq C_\eta^p\|C\|^p\bigg(\int_0^t(t-r)^{\frac{(\eta-1)p}{p-1}}dr\bigg)^{p-1}\bigg(\int_0^tE\|v^\lambda(r,z(r))-v^\lambda(r,w(r))\|^p_{H^\beta}\bigg)\\
&\leq C_\eta^p\|C\|^pT^{(p\eta-1)}\bigg[\frac{p-1}{\eta p-1}\bigg]^{p-1}\bigg(\int_0^tE\|v^\lambda(r,z(r))-v^\lambda(r,w(r))\|^p_{H^\beta}\bigg)\\
&\leq C_\eta^p\|C\|^p\frac{C_v}{\lambda^p}T^{(p\eta)}\bigg[\frac{p-1}{\eta p-1}\bigg]^{p-1}\bigg(\int_0^tE\|z(r)-w(r)\|^p_{H^\beta}dr\bigg)\\
\end{align*}
and
\begin{align*}
J_3&=E\|\int_0^t(t-r)^{\eta-1}M_{\eta,\eta}(t-r)[\hbar(r,z(r))-\hbar(r,w(r))]dW(r)\|^p_{H^{\beta}}\\
&\leq \kappa(p)C_\eta^pE\bigg(\int_0^t\|(t-r)^{\eta-1}M_{\eta,\eta}(t-r)A_{\beta}[\hbar(r,z(r))-\hbar(r,w(r))]\|^2_{L_0^2}dr\bigg)^{p/2}\\
&\leq \kappa(p)C_\eta^p\bigg(\int_0^t(t-r)^{\frac{(\eta-1)2p}{p-2}}dr\bigg)^{\frac{p-2}{2}}\bigg(\int_0^tE\|A_{\beta}[\hbar(r,z(r))-\hbar(r,w(r))]\|^p_{L_0^2}dr\bigg)\\
&\leq \kappa(p)C_\eta^pL_2^pT^{\frac{2p\eta-p-2}{2}}\bigg[\frac{p-2}{2\eta p-p-2}\bigg]^{\frac{p-2}{2}}\bigg(\int_0^tE\|z(r)-w(r)\|^p_{H^\beta}\bigg).
\end{align*}
Hence
\begin{align*}
E\|(&\mathcal{F}_\lambda z)(t)-(\mathcal{F}_\lambda w)(t)\|^p_{H^\beta}\\
&\leq 3^{p-1}\bigg[C_\eta^p C_2^p\bigg(\max_{t\in[0,T]}E[\|z(t)\|^p_{H^\beta}]+\max_{t\in[0,T]}E[\|w(t)\|^p_{H^\beta}]\bigg)\bigg[\frac{p-1}{p\eta[1-\frac{1}{\alpha}]-1}\bigg]^{p-1}T^{p\eta[1-\frac{1}{\alpha}]-1}\\
&\quad+C_\eta^p\|C\|^p\frac{C_v}{\lambda^p}T^{(p\eta)}\bigg[\frac{p-1}{\eta p-1}\bigg]^{p-1}+\kappa(p)C_\eta^pL_2^pT^{\frac{2p\eta-p-2}{2}}\bigg[\frac{p-2}{2\eta p-p-2}\bigg]^{\frac{p-2}{2}}\bigg]\\
&\quad\quad\quad\quad\times\int_0^tE\|z(r)-w(r)\|^p_{H^\beta}dr.
\end{align*}
As a result, there is a positive number $\Phi(\lambda)$ such that
\begin{align*}
E\|(&\mathcal{F}_\lambda z)(t)-(\mathcal{F}_\lambda w)(t)\|^p_{H^\beta}\leq \Phi(\lambda)\int_0^tE\|z(r)-w(r)\|^p_{H^\beta}dr.
\end{align*}
Iterating, we see that, for any natural number $n\geq 1$,
\begin{align*}
E\|(&\mathcal{F}_\lambda^n z)(t)-(\mathcal{F}_\lambda^n w)(t)\|^p_{H^\beta}\leq \frac{(T\Phi(\lambda))^n}{n!}\int_0^tE\|z(r)-w(r)\|^p_{H^\beta}dr
\end{align*}
for any $\lambda>0$. When $n$ is sufficiently large (so that $\frac{(T\Phi(\lambda))^n}{n!}<1$),  using Banach's fixed point theorem, we obtain that the operator $\mathcal{F}_\lambda^n$ has a unique fixed point
$u_{\lambda} \in L^p(\Sigma, H^\beta)$.
Since $$\mathcal{F}_\lambda^n(\mathcal{F}_\lambda(u_{\lambda})) = \mathcal{F}_\lambda(\mathcal{F}_\lambda^n(u_{\lambda})) = \mathcal{F}_\lambda(u_{\lambda}),$$
and $\mathcal{F}_\lambda^n$ has a unique fixed point, it follows that $\mathcal{F}_\lambda(u_{\lambda})$ has a unique fixed point i.e.
$\mathcal{F}_\lambda(u_{\lambda})$ is  the unique fixed point of $F_{\lambda}^n$, that is, $F_{\lambda}(u{\lambda}) = u_{\lambda}$,
which is a mild solution to (\ref{BE2}).
\end{proof}

\begin{rmk}
Theorem \ref{MainThm} proves that under the assumptions (1) and (2), the time-fractional SNSE (\ref{BE2}) has a unique mild solution $(z(t))_{t\in[0,T]}$ of the form (\ref{mild}).
\end{rmk}

\begin{theorem}\label{Thmctrl}
If conditions $(1)-(3)$ hold true and the functions $\hbar$ and $G$ are uniformly bounded, then the time-fractional stochastic Navier-Stokes equation (\ref{BE2}) is approximately controllable.
\end{theorem}
\begin{proof}
Let $z_{\lambda}$ be the fixed point of $\mathcal{F}_\lambda$. Using the stochastic Fubini theorem, one can show that the fixed point of $\mathcal{F}_\lambda$ satisfies
\begin{align*}
z_\lambda(T)&=z_T-\lambda(\lambda I+\Upsilon_0^T)^{-1}[Ez_T-M_\eta(T)z_0]+\int_0^T\lambda(\lambda I+\Upsilon_0^T)^{-1}(T-r)^{\eta-1}M_{\eta,\eta}(T-r)G(z_\lambda(r))dr\\
&\quad+\int_0^T\lambda(\lambda I+\Upsilon_0^T)^{-1}[(T-r)^{\eta-1}M_{\eta,\eta}(T-r)\hbar(r,z_\lambda(r))-\varphi(r)]dW(r).
\end{align*}
Since the functions $\hbar$ and $G$ are uniformly bounded, there exist constants $D_1>0$ and $D_2>0$ so that
\begin{align*}
\|\hbar(r,z_\lambda(r))\|^p\leq D_1,
\end{align*}
and
\begin{align*}
\|G(z_\lambda(r))\|^p\leq D_2.
\end{align*}
Therefore we can find a subsequence $\{\hbar(c,z_\lambda(c)), G(z_\lambda(c))\}$ which converges weakly to $\{\hbar(r,z(r)), G(z(r))\}$.\\
Using the Lebesgue dominated convergence theorem, we obtain

\begin{align*}
E\int_0^T\|M_{\eta,\eta}(T-r)[\hbar(r,z_\lambda(r))-\hbar(r,z(r))]\|^pdr\rightarrow 0
\end{align*}
and
\begin{align*}
E\int_0^T\|M_{\eta,\eta}(T-r)[G(z_\lambda(r))-G(z(r))]\|^pdr\rightarrow 0.
\end{align*}
Hence
\begin{align}\label{control1}
E\|z_\lambda(T)-z_T\|^p\leq &6^{p-1}\bigg[\|\lambda(\lambda I+\Upsilon_0^T)^{-1}[Ez_T-M_\eta(T)z_0]\|^p\nonumber\\
&+E\bigg(\int_0^T(T-r)^{\eta-1}\|\lambda(\lambda I+\Upsilon_0^T)^{-1}\varphi(r)\|^p_{L^2_0}dr\bigg)\nonumber\\
&+E\bigg(\int_0^T(T-r)^{\eta-1}\|\lambda(\lambda I+\Upsilon_0^T)^{-1}\||M_{\eta,\eta}(T-r)[G(z_\lambda(r))-G(z(r))]\|dr\bigg)^p\nonumber\\
&+E\bigg(\int_0^T(T-r)^{\eta-1}\|\lambda(\lambda I+\Upsilon_0^T)^{-1}\|\|M_{\eta,\eta}(T-r)G(z(r))\|dr\bigg)^p\nonumber\\
&+E\bigg(\int_0^T\|(T-r)^{\eta-1}\lambda(\lambda I+\Upsilon_0^T)^{-1}M_{\eta,\eta}(T-r)[\hbar(r,z_\lambda(r))-\hbar(r,z(r))]\|^2_{L^2_0}dr\bigg)^{\frac{p}{2}}\nonumber\\
&+E\bigg(\int_0^T\|(T-r)^{\eta-1}\lambda(\lambda I+\Upsilon_0^T)^{-1}M_{\eta,\eta}(T-r)\hbar(r,z(r))\|^2_{L^2_0}dr\bigg)^{\frac{p}{2}}\bigg]
\end{align}
From assumption (3), it follows that the operator $$\lambda(\lambda I+\Upsilon_0^T)^{-1}\rightarrow 0 \; \text{strongly as} \; \lambda\rightarrow 0.$$ In addition, $$\|\lambda(\lambda I+\Upsilon_0^T)^{-1}\|\leq1.$$ Thus, using the Lebesgue dominated convergence theorem and the compactness of the operator $M_{\eta,\eta}$, we obtain $$E\|z_\lambda(T)-z_T\|^p\rightarrow 0 \;\text{as} \;\lambda\rightarrow 0.$$ Hence using the definition of approximate controllability, the time-fractional SNSE (\ref{BE2}) is approximately controllable on $[0,T]$,  as asserted.
\end{proof}

\begin{rmk}
   Theorem \ref{Thmctrl} proves that under assumptions (1)-(3), there exists a control function that steers the solution of the time-fractional stochastic Navier-Stokes equation (\ref{BE2}) from the initial state $z_0$ to the neighborhood of the final state $z_T$ in $[0,T]$.
\end{rmk}

\section{Example}

To demonstrate the applicability of the theoretical results on approximate controllability, we consider the following two-dimensional time-fractional stochastic Navier-Stokes equation in the square domain $\mathcal{O}=(0,1)^2$ with Dirichlet boundary conditions: 
\begin{align}\label{Exm}
\begin{split}
\partial_t^\eta z(t,x,y) + \nu(-\Delta)^{\alpha/2} z(t,x,y) 
&= C v(t,x,y) + \sigma z(t,x,y) \frac{dW(t)}{dt}, \quad \forall \, t\in (0,1]\\
z(0,x,y) = z_0(x,y)&=  \sin(\pi x)\sin(\pi y),
\end{split}
\end{align}

where $z(t,x,y)$ denotes the unknown variable, $\eta = 0.7$ is the fractional time derivative order, $\alpha = 1.8$ is the fractional Laplacian order, $\nu = 0.1$ is the viscosity coefficient,  $\sigma = 0.05$ is the noise intensity, the control operator $C=I$ and $W(t)$ is a standard Wiener process. The control input is

\[
 v(t,x,y) =\sin(\pi x)\sin(\pi y), \quad t \in [0,1].
\]
The target state is the velocity field pattern induced by the control input, i.e., the first Fourier mode $\sin(\pi x)\sin(\pi y)$. 

Setting $A=\nu(-\Delta)^{\alpha/2}$, the mild solution of (\ref{Exm}) is given by
\[
z(t)=E_{\eta}(-A t^{\eta})z_0+\int_0^t (t-s)^{\eta-1}E_{\eta,\eta}(-A (t-s)^{\eta}) v(s)\,ds
+\sigma\int_0^t (t-s)^{\eta-1}E_{\eta,\eta}(-A (t-s)^{\eta}) z(s)\,dW(s).
\]
Here $E_{\eta,\beta}$ denotes the Mittag--Leffler function and the last term is an Itô stochastic convolution. Under the assumptions stated in Section~2, analyticity of $A$, admissibility of $C$, linear growth and Lipschitz condition on the noise coefficient, this integral equation admits a unique mild solution in $L^{2}(\Omega;C([0,T];H))$.

\subsection{Numerical Scheme}

\begin{enumerate}
    \item \textbf{Fractional derivative:} We approximated the fractional derivative using the Grünwald–Letnikov formula with recursive computation of weights:
    \[
    w_0 = 1, \quad w_k = \left(1 - \frac{\eta+1}{k}\right) w_{k-1}, \quad k \ge 1.
    \]
    
    \item \textbf{Fractional Laplacian:} We approximated the fractional Laplacian spectrally. For a velocity field $z(x,y)$ expanded in sine basis, the fractional Laplacian is applied in Fourier space:
    \[
    (-\Delta)^{\alpha/2} z(x,y) = \sum_{m,n} (\pi^2(m^2+n^2))^{\alpha/2} \hat{z}_{m,n} \sin(m \pi x) \sin(n \pi y),
    \]
    where $\hat{z}_{m,n}$ are the Fourier coefficients.
    
    \item \textbf{Noise term:} Discretized as
    \[
    \sigma z^n \frac{\Delta W^n}{\Delta t}, \quad \Delta W^n \sim \mathcal{N}(0, \Delta t),
    \]
    with independent realizations at each spatial point and time step.
    
    \item \textbf{Time integration:} Semi-implicit Euler scheme is used:
    \[
    z^{n+1} = z^n + \Delta t \big(-\nu (-\Delta)^{\alpha/2} z^n + v^n + \text{noise} - \text{fractional derivative term}\big).
    \]
\end{enumerate}

\subsection{Results and Interpretation}

Figure~\ref{fig:L2norm} shows the evolution of the $L^2$ norm of the velocity field $\|z(t,\cdot,\cdot)\|_{L^2}$ with and without control over time $t\in [0,1]$. This clearly demonstrates that the controlled trajectory of the velocity field approaches the desired target pattern $\sin(\pi x)\sin(\pi y)$ over time. Despite the presence of stochastic perturbations and memory effects, the applied control successfully guides the system close to the intended state. This illustrates the approximate controllability of the time-fractional stochastic Navier–Stokes system, confirming the theoretical results derived in Section 3.

\begin{figure}[h!]
    \centering
    \includegraphics[width=0.7\textwidth]{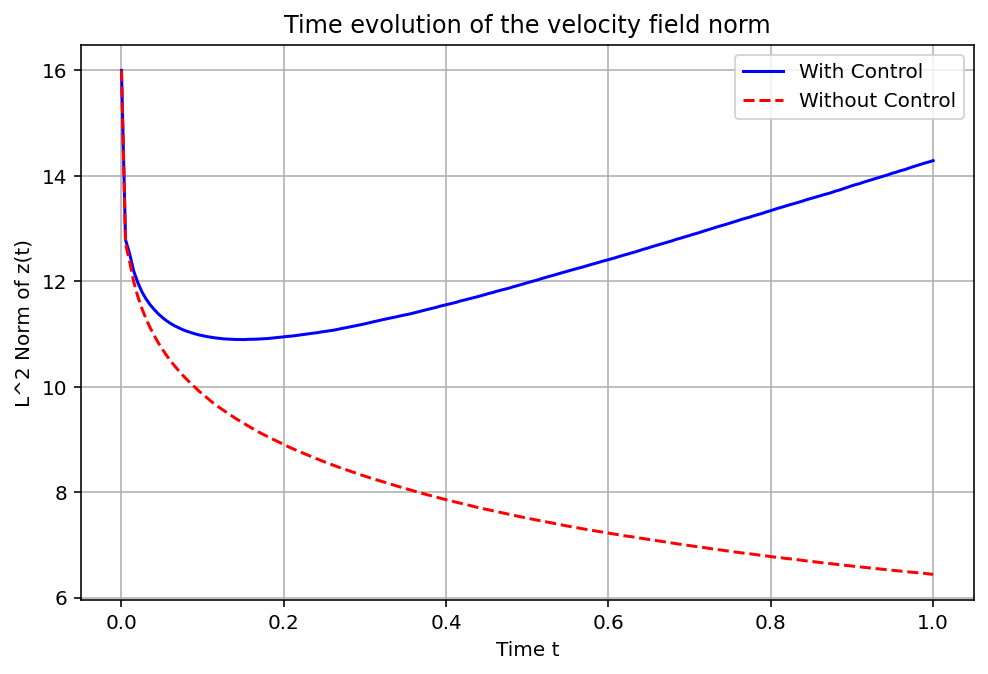}
    \caption{Time evolution of the $L^2$ norm of the velocity field $\|z(t,\cdot,\cdot)\|_{L^2}$ under the controlled and uncontrolled stochastic fractional Navier–Stokes dynamics.}
    \label{fig:L2norm}
\end{figure}

\text{}

\noindent {\bf  Declarations}\\
 
\text{}
\textbf{{\small
 Funding}} Simeon Reich was partially supported by the Israel Science Foundation (Grant 820/17), the Fund for the Promotion of Research at the Technion (Grant 2001893) and by the Technion General Research Fund (Grant 2016723). The work of J.J. Nieto has been partially supported by the Agencia Estatal de Investigación (Spain), Grant PID2020-113275GB-I00 funded by MCIN/AEI/ 10.13039/501100011033 and by “ERDF A way of making Europe”, as well as Xunta de Galicia grant ED431C 2019/02 for Competitive Reference Research Groups (2019-22).\\

\noindent 
\text{}
\textbf{{\small
 Conflict of interest}} The authors declare that they have no conflict of interest.

\bibliographystyle{plain}
\bibliography{sample}

\end{document}